\documentclass[letterpaper, 10pt, conference]{ieeeconf}
\pdfoutput=1\IEEEoverridecommandlockouts 
\usepackage{amsmath,amssymb,url}
\usepackage{graphicx,subfigure}
\usepackage[pdfpagemode=none,pdfstartview=FitH]{hyperref}

\newcommand{\bracket}[1]{\ensuremath{\left[ #1 \right]}}
\newcommand{\braces}[1]{\ensuremath{\left\{ #1 \right\}}}

\newcommand{\refeqn}[1]{(\ref{eqn:#1})}

\newcommand{\tr}[1]{\mathrm{tr}\ensuremath{\negthickspace\bracket{#1}}}
\newcommand{\trs}[1]{\mathrm{tr}\ensuremath{[#1]}}

\newcommand{\SO}{\ensuremath{\mathsf{SO(3)}}}
\newcommand{\T}{\ensuremath{\mathsf{T}}}

\newcommand{\so}{\ensuremath{\mathfrak{so}(3)}}

\renewcommand{\Re}{\ensuremath{\mathbb{R}}}

\newcommand{\D}{\ensuremath{\mathbf{D}}}
\newcommand{\Sph}{\ensuremath{\mathsf{S}}}

\title{\LARGE \bf 
Vision-Based Relative Attitude Formation Tracking}

\author{Tse-Huai Wu and Taeyoung Lee\authorrefmark{1}%
\thanks{Tse-Huai Wu and Taeyoung Lee, Mechanical and Aerospace Engineering, The George Washington University, Washington DC 20052. {\tt \{wu52,tylee\}@gwu.edu}}%
\thanks{\textsuperscript{\footnotesize\ensuremath{*}}This research has been supported in part by NSF under the grant CMMI-1243000 (transferred from 1029551).}
}

\newtheorem{prop}{Proposition}
\newtheorem{assump}{Assumption}

\begin{document}
\allowdisplaybreaks
\maketitle \thispagestyle{empty} \pagestyle{empty}

\begin{abstract}
Relative attitude formation control systems are developed for multiple spacecraft, based on the line-of-sight measurements between spacecraft in formation. The proposed control systems are unique in the sense that they do not require constructing the full attitudes of spacecraft and comparing them to obtain the relative attitudes indirectly. Instead, the control inputs are directly expressed in terms of line-of-sight measurements to control relative attitude formation precisely and efficiently. It is shown that the relative attitudes almost globally asymptotically track their desired relative attitudes. The desirable properties are illustrated by a numerical example.
\end{abstract}

\section{Introduction}

The coordinated control of multiple spacecraft in formation has been widely studied, as there are distinct advantages~\cite{FaxMurITAC04,JadLinITAC03}. 
However, successful operation of multiple spacecraft requires more sophisticated control systems. For example, for interferometer missions like Darwin, spacecraft in formation should maintain a specific relative position and attitude configuration precisely. Relative position control and estimation have been well-addressed, by using Carrier-phase Differential GPS for precise relative navigation~\cite{Mit04,GarChaPIAC05}.


Noticeable contributions on relative attitude control may be divided into leader-follower strategy~\cite{KanYehIJRNC02,NijRod03}, behavior-based control~\cite{BalArkITRA98,BeaLawITCST01}, and virtual structures~\cite{RenBeaPICCTA04,RenBeaPAGNCC02}. The aforementioned control systems for spacecraft attitude formation control have distinct features, but all of them are based on a common framework: the absolute attitude of each spacecraft with respect to an inertial frame is measured independently by using a local inertial measurement unit, and those measurements are transmitted to other spacecraft to determine relative attitudes by comparison.

This causes restrictions on the performance of coordinated spacecraft. First, \textit{all} of spacecraft should be equipped with possibly expensive hardware systems to determine the absolute attitude completely. This may increase the overall cost of development significantly. Second, attitude formation is \textit{indirectly} controlled by comparing the absolute attitudes of multiple spacecraft in the formation. This results in a fundamental limitation on the accuracy of attitude formation control systems, since measurement errors of multiple sensors are accumulated when determining relative attitudes. 

Vision-based sensors have been widely applied for navigation of autonomous vehicles, where low-cost optical sensors are used to extract visual features to localize a vehicle~\cite{DesKakITPAMI02}. In particular, it has been shown that line-of-sight (LOS) measurements between spacecraft in formation determine the relative attitudes completely. An extended Kalman filter for relative attitude is developed based on LOS observations~\cite{KimCraJGCD07}. The LOS measurements are also used for relative attitude determination of multiple vehicles~\cite{AndCraJGCD09,LinCraJGCD11}.

In this paper, a relative attitude formation control scheme is developed based on LOS measurements. Spacecraft in formation measure the LOS toward other spacecraft such that relative attitude between them asymptotically track a given desired relative attitude. Compared to other spacecraft attitude formation control systems, the proposed relative attitude control systems is unique in the sense that control inputs are directly expressed in terms of LOS measurements, and it does not require determining the full absolute attitude of spacecraft in formation or the full relative attitude between them. Therefore, relative attitudes are directly controlled, while utilizing the desirable features of vision-based sensors, which have higher accuracies at a relatively low cost, and they also have long-term stability, requiring no corrections in measurements as opposed to gyros. 

Compared with the preliminary work for relative attitude stabilization between two spacecraft~\cite{LeePACC12}, the control system proposed in this paper requires extensive analyses to take into full consideration of stability of time-varying systems for tracking, and the network structures between multiple spacecraft. The paper also provides stronger exponential stability, and numerical simulations with image processing.

Another distinct feature of the proposed relative attitude control system is that it is constructed on the special orthogonal group, $\SO$. Attitude control systems developed on minimal representations, such as Euler-angles, have singularities, and therefore their performance for large angle rotational maneuvers is severely limited. Quaternions do not have singularities, the ambiguity in representing attitude should be carefully resolved. By following geometric control approaches~\cite{ChaSanICSM11,BulLew05}, the proposed control system is developed in a coordinate-free fashion, and it does not have any singularity or ambiguity. 


\section{Problem Formulation}\label{sec:PF}

\subsection{Spacecraft Attitude Formation Configuration}

Consider an arbitrary number $n$ of spacecraft in formation. Each spacecraft is considered as a rigid body, and an inertial reference frame and body-fixed frames are defined. The attitude of each spacecraft is the orientation of its body-fixed frame with respect to the inertial reference frame, and it is represented by a rotation matrix in the special orthogonal group, namely
\begin{align*}
\SO =\{ R\in\Re^{3\times 3}\,|\, R^\T R =I,\quad \det{R}=1\}.
\end{align*}
Each spacecraft measures the LOS from itself toward the other assigned spacecraft. A LOS observation is represented by a unit vector in the two-sphere,  defined as
\begin{align*}
\Sph^2 = \{ s\in\Re^3\,|\, \|s\|=1\}.
\end{align*}

For $i,j\in\{1,\ldots,n\}$ and $i\neq j$, define
\begin{center}
\begin{tabular}{lp{5.8cm}}
{$R_i\in\SO$} & the absolute attitude for the $i$-th spacecraft, representing the linear transformation from the $i$-th body-fixed frame to the inertial reference frame,\\
{$s_{ij}\in\Sph^2$} & the unit vector toward the $j$-th spacecraft from the $i$-th spacecraft, represented in the inertial frame,\\
{$b_{ij}\in\Sph^2$} & the LOS direction observed from the $i$-th spacecraft to the $j$-th spacecraft, represented in the $i$-th body fixed frame,\\
{$Q_{ij}\in\SO$} & the relative attitude of the $i$-th spacecraft with respect to the $j$-th spacecraft,\\
{$Q_{ij}^d\in\SO$} & the desired relative attitude for $Q_{ij}$.
\end{tabular}
\end{center}
According to these definitions, the directions of the relative positions $s_{ij}$ in the inertial reference frame are related to the LOS observation $b_{ij}$ in the $i$-th body-fixed frame as follows:
\begin{gather}
s_{ij} = R_{i}b_{ij},\quad b_{ij} = R_{i}^\T s_{ij}\label{eqn:sbbs}.
\end{gather}
In short, $b_{ij}$ represents the LOS observation of $s_{ij}$, observed from the $i$-th body. The relative attitude is given by
\begin{align}
Q_{ij}= R_j^\T R_i,\label{eqn:Qij}
\end{align}
which represents the linear transformation of the representation of a vector from the $i$-th body fixed frame to the $j$-th body-fixed frame. Note that $Q_{ij}=Q_{ji}^\T$.

To assign a set of LOS that should be measured for each spacecraft, a graph $(\mathcal{N},\mathcal{E})$ is defined as follows. Each spacecraft is considered as a node, and the \textit{set of nodes} is given by $\mathcal{N}=\{1,\ldots,n\}$. The \textit{set of edges} $\mathcal{E}\subset\mathcal{N}\times\mathcal{N}$ is defined such that the relative attitude between the $i$-th spacecraft and the $j$-th spacecraft is directly controlled if $(i,j)\in\mathcal{E}$. It is undirected, i.e., $(i,j)\in\mathcal{E} \Leftrightarrow (j,i)\in\mathcal{E}$. For each pair of two spacecraft in the edge set, another third spacecraft is assigned by the \textit{assignment map} $\rho:\mathcal{E}\rightarrow\mathcal{N}$. As the edge set is undirected, the assignment map is symmetric, i.e., $\rho(i,j)=\rho(j,i)$.

For convenience, the edge set and the image of the assignment map are combined to form the \textit{assignment set}:
\begin{align}
\mathcal{A} =\{(i,j,k)\in \mathcal{E}\times\mathcal{N} \,|\, (i,j)\in\mathcal{E}, k=\rho(i,j)\}.
\label{eqn:A}
\end{align}
Let the \textit{measurement set} $\mathcal{L}_i$ be the set of LOS measured from the $i$-th spacecraft, and let the \textit{communication set} $\mathcal{C}_{ij}$ be the LOS transferred from the $i$-th spacecraft to the $j$-th spacecraft. 

\setlength{\unitlength}{0.1\columnwidth}
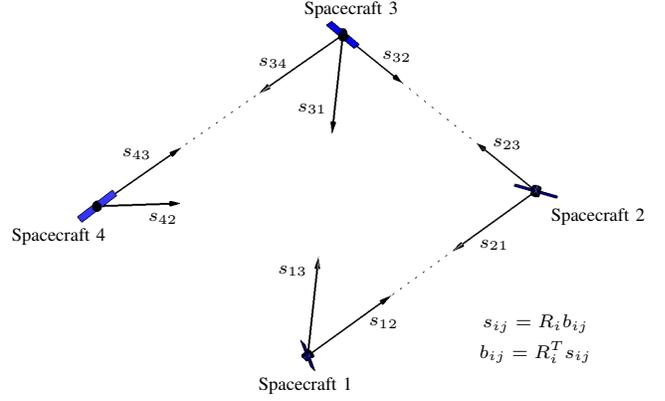
\begin{figure}
\scriptsize\selectfont
\centerline{
\begin{picture}(10,5.8)(-0.2,0)
\put(0.7,0){\includegraphics[width=0.75\columnwidth]{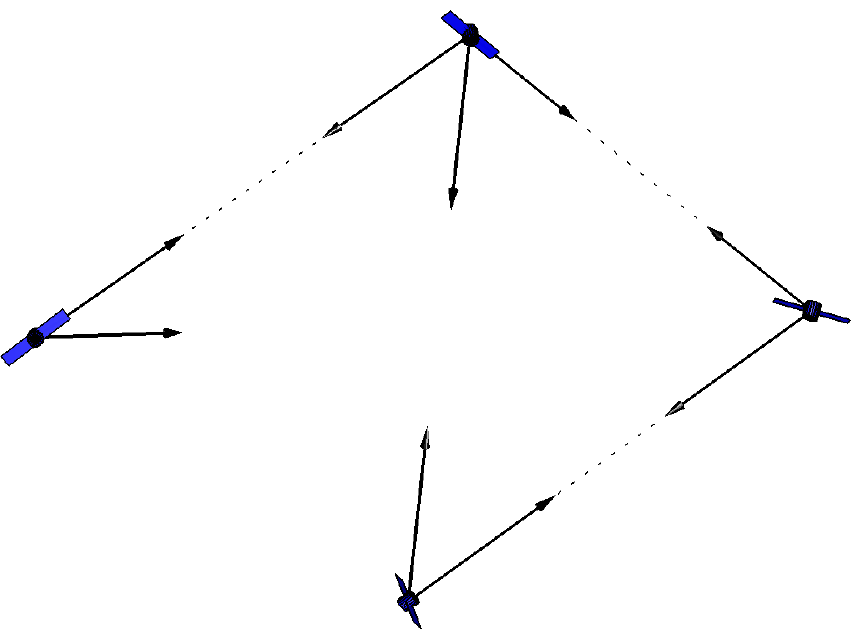}}
\put(-0.3,2.0){Spacecraft 4}
\put(4.2,5.5){Spacecraft 3}
\put(8.0,2.3){Spacecraft 2}
\put(3.5,-0.3){Spacecraft 1}
\put(1.4,3.3){$s_{43}$}
\put(1.8,2.3){$s_{42}$}
\put(3.5,4.7){$s_{34}$}
\put(4.1,4.0){$s_{31}$}
\put(5.4,4.8){$s_{32}$}
\put(7.1,3.45){$s_{23}$}
\put(6.9,1.9){$s_{21}$}
\put(3.8,1.5){$s_{13}$}
\put(5.2,0.7){$s_{12}$}
\put(6.9,0.2){\shortstack[c]{$s_{ij}=R_ib_{ij}$\\$b_{ij}=R_i^T s_{ij}$}}
\end{picture}}
\vspace*{0.2cm}
\caption{Formation of four spacecraft: the direction along the relative position of the $i$-th body from the $j$-th body is denoted by $s_{ij}$ in the inertial reference frame. The LOS observation of $s_{ij}$ with respect the $i$-th body fixed frame, namely $b_{ij}$ is obtained from \refeqn{sbbs}.}\label{fig:FS}\end{figure}

\begin{assump}\label{assump:fixeds}
The configuration of the relative positions is fixed, i.e., $\dot s_{ij}=0$ for all $i,j\in\mathcal{N}$ with $i\neq j$.
\end{assump}

\begin{assump}\label{assump:plane}
The third spacecraft assigned to each edge does not lie on the line joining two spacecraft connected by the edge, i.e., $s_{ik}\times s_{jk}\neq 0$ for every $(i,j,k)\in\mathcal{A}$.
\end{assump}

\begin{assump}\label{assump:LOS}
The measurement set of the $i$-th spacecraft is
\begin{align}
\mathcal{L}_i =\{ b_{ij},b_{ik}\in\Sph^2 \,|\, (i,j,k)\in\mathcal{A}\}.\label{eqn:Li}
\end{align}
\end{assump}

\begin{assump}\label{assump:Comm}
The communication set from the $i$-th spacecraft to the $j$-th spacecraft is given by
\begin{align}
\mathcal{C}_{ij} =
\begin{cases}
\{ b_{ij},b_{i\rho(i,j)}\} & \text{if $(i,j)\in\mathcal{E}$},\\
\emptyset & \text{otherwise}.
\end{cases}\label{eqn:Cij}
\end{align}
\end{assump}

\begin{assump}\label{assump:chain}
In the edge set, spacecraft are paired serially by daisy-chaining.
\end{assump}

The first assumption reflects the fact that this paper does not consider the translational dynamics of spacecraft, and we focus on the rotational attitude dynamics only. 
The proposed control input does not depend on the values of $s_{ij}$, but its stability analyses is based on the first assumption saying that $s_{ij}$ is fixed. The second assumption is required to determine the relative attitude between two spacecraft paired in the edge set from the assigned LOS measurements. The third assumption states that each spacecraft measures the LOS toward the paired spacecraft in the edge set, and the LOS toward the third spacecraft assigned to each pair by the assignment map. The fourth assumption implies that a spacecraft communicate only with the spacecraft paired with itself. The last assumption is made to simplify stability analysis, and the proposed relative attitude formation control system can be extended for other network topologies. 

An example for formation of four spacecraft satisfying these assumptions are illustrated at Figure \ref{fig:FS}, where
\begin{align*}
\mathcal{A} = \{ (1,2,3), (2,1,3),\, (2,3,1), (3,2,1),\, (3,4,2), (4,3,2)\}.
\end{align*}
The measurement sets and the communication sets can be determined by \refeqn{Li} and \refeqn{Cij} from $\mathcal{A}$. For example, for the third spacecraft, we have $\mathcal{L}_3 =\{ b_{31}, b_{32}, b_{34}\}$, $\mathcal{C}_{32} = \{b_{32}, b_{31}\}$, and $\mathcal{C}_{34} = \{b_{34}, b_{32}\}$.

\subsection{Spacecraft Attitude Dynamics}

The equations of motion for the attitude dynamics of each spacecraft are given by
\begin{gather}
J_i\dot \Omega_i + \Omega_i\times J_i\Omega_i = u_i,\label{eqn:Wdot}\\
\dot R_i = R_i\hat\Omega_i,\label{eqn:Rdot}
\end{gather}
where $J_i\in\Re^{3\times 3}$ is the inertia matrix of the $i$-th spacecraft, and $\Omega_i\in\Re^3$ and $u_i\in\Re^3$ are the angular velocity and the control moment of the $i$-th spacecraft, represented with respect to its body-fixed frame, respectively. 

The \textit{hat} map $\wedge :\Re^{3}\rightarrow\so$ transforms a vector in $\Re^3$ to a $3\times 3$ skew-symmetric matrix such that $\hat x y = x\times y$ for any $x,y\in\Re^3$. 
The inverse of the hat map is denoted by the \textit{vee} map $\vee:\so\rightarrow\Re^3$.
Few properties of the hat map are summarized as follows:
\begin{gather}
    \widehat{x\times y} = \hat x \hat y -\hat y \hat x = yx^T-xy^T,\label{eqn:hatxy}\\
    \tr{\hat x A}=
\frac{1}{2}\tr{\hat x (A-A^T)}=-x^T (A-A^T)^\vee,\label{eqn:hatxA}\\
R\hat x R^T = (Rx)^\wedge,\label{eqn:RxR}\\
x\cdot \hat y z = y\cdot \hat z x = z\cdot \hat x y,\label{eqn:xyz}\\
\hat x\hat y z = x\times (y\times z) = (x\cdot z) y -(x\cdot y) z,\label{eqn:STP}\\
\hat x\hat y z -\hat z\hat y x=\hat y \hat x z,\label{eqn:STP2}
\end{gather}
for any $x,y,z\in\Re^3$, $A\in\Re^{3\times 3}$, and $R\in\SO$. Throughout this paper, the 2-norm of a matrix $A$ is denoted by $\|A\|$, and the dot product of two vectors is denoted by $x \cdot y = x^Ty$. The maximum eigenvalue and the minimum eigenvalue of $J_i$ are denoted by $\lambda_{M_i}$ and $\lambda_{m_i}$, respectively.

\section{Relative Attitude Tracking Between Two Spacecraft}\label{sec:RAT2}

We first consider a simpler case of controlling the relative attitude between two spacecraft. Based on the results of this section, relative attitude formation control systems are developed later. As a concrete example, we develop a control system for the relative attitude between Spacecraft 1 and Spacecraft 2, namely $Q_{12}=R_2^\T R_1$ illustrated at Figure \ref{fig:FS}. The corresponding edge set, assignment set and measurement sets used in \textit{this section} are given by
\begin{gather}
\mathcal{E} =\{(1,2),(2,1)\},\quad \mathcal{A} = \{ (1,2,3), (2,1,3)\},\label{eqn:Ap}\\
\mathcal{L}_1=\mathcal{C}_{12}=\{b_{12},b_{13}\},\;
\mathcal{L}_2=\mathcal{C}_{21}=\{b_{21},b_{23}\}.
\end{gather}

Suppose that a desired relative attitude $Q_{12}^d(t)$ is given as a smooth function of time. It satisfies the kinematic equation:
\begin{align}
\dot Q_{12}^d = Q_{12}^d \hat\Omega_{12}^d,\label{eqn:dotQd12}
\end{align}
where $\Omega_{12}^d$ is the desired relative angular velocity. Note that these also yield $Q_{21}^d = (Q_{12}^d)^\T$ from \refeqn{Qij}, and it satisfies
\begin{align}
\dot Q_{21}^d = Q_{21}^d \hat\Omega_{21}^d, \label{eqn:dotQd21}
\end{align}
where $\Omega_{21}^d = - Q^d_{12}\Omega_{12}^d$.

The goal is to design control inputs $u_1,u_2$ in terms of the LOS measurements in $\mathcal{L}_1\cup \mathcal{L}_2$ such that $Q_{12}$ asymptotically follows $Q_{12}^d$, i.e., $Q_{12}(t)\rightarrow Q_{12}^d(t)$ as $t\rightarrow \infty$.

\subsection{Kinematics of Relative Attitudes and Lines-of-Sight}

For any $i,j\in\mathcal{N}$, the time-derivative of the relative attitude is given, from \refeqn{Rdot}, by
\begin{align}
\dot Q_{ij} & = -\hat\Omega_j R_j^\T R_i + R_j^\T R_i\hat\Omega_i = Q_{ij}\hat\Omega_i - \hat\Omega_j Q_{ij}\nonumber\\
& = Q_{ij} (\Omega_i - Q_{ij}^T \Omega_j)^\wedge \triangleq Q_{ij}\hat \Omega_{ij},\label{eqn:dotQij}
\end{align}
where the relative angular velocity $\Omega_{ij}\in\Re^3$ of the $i$-th spacecraft with respect to the $j$-th spacecraft is defined as
\begin{align}
\Omega_{ij} = \Omega_i - Q_{ij}^\T \Omega_j.\label{eqn:Wij}
\end{align}

From \refeqn{sbbs} and \refeqn{Rdot}, the time-derivative of the LOS measurement $b_{ij}$ is given by
\begin{align}
\dot b_{ij} = \dot{R}_i^\T s_{ij} = -\hat\Omega_i R_i^\T s_{ij} = b_{ij}\times\Omega_i.\label{eqn:dotbij}
\end{align}
Let $b_{ijk}\in\Re^3$ be $b_{ijk}=b_{ij}\times b_{ik}$. From \refeqn{dotbij} and \refeqn{STP}, it can be shown that
\begin{align}
\dot b_{ijk} & = (b_{ij}\times \Omega_i)\times b_{ik} + b_{ij}\times (b_{ik}\times\Omega_i)\nonumber\\
& = -(\Omega_i\cdot b_{ik})b_{ij} + (\Omega_i\cdot b_{ij})b_{ik} 
 = b_{ijk}\times \Omega_i.
\end{align}

\subsection{Relative Attitude Tracking}\label{sec:RAT}

It has been shown that four LOS measurements in $\mathcal{L}_1\cup \mathcal{L}_2$ completely determine the relative attitude $Q_{12}$ from the following constraints~\cite{LeePACC12}:
\begin{gather}
b_{12} = - Q_{12}^\T b_{21},\label{eqn:consta}\\
\frac{b_{123}}{\| b_{123} \|} = - \frac{Q_{12}^\T b_{213}}{\| b_{213} \|}
.\label{eqn:constb}
\end{gather}
These are derived from the fact that four unit vectors, namely $\{s_{12},s_{13},s_{21},s_{23}\}$ lie on the sides of a triangle composed of three spacecraft. The first constraint \refeqn{consta} states that the unit vector from Spacecraft 1 to Spacecraft 2 is exactly opposite to the unit vector from Spacecraft 2 to Spacecraft 1, i.e., $s_{12}=-s_{21}$. The second constraint \refeqn{constb} implies that the plane spanned by $s_{12}$ and $s_{13}$ should be co-planar with the plane spanned by $s_{21}$ and $s_{23}$. These geometric constraints are simply expressed with respect to the first body-fixed frame to obtain \refeqn{consta} and \refeqn{constb}. The relative attitude $Q_{12}$ is determined uniquely by the LOS measurements $\{b_{12},b_{13},b_{21},b_{23}\}$ according to \refeqn{consta} and \refeqn{constb}.

We develop a relative attitude tracking control system based on these two constraints. More explicitly, control inputs are chosen such that two constraints are satisfied when the relative attitude is equal to its desired value. As both constraints are conditions on unit vectors, controller design similar to tracking control on the two-sphere. From now on, variables related to the first constraint \refeqn{consta} (resp., the second constraint \refeqn{constb}) are denoted by the sub- or super-script $\alpha$ (resp., $\beta$).

First, configuration error functions are defined as
\begin{align}
\Psi_{12}^\alpha & = \frac{1}{2}\|b_{21}+Q_{12}^d b_{12}\|^2 = 1 +  b_{21}\cdot Q_{12}^d b_{12},\\
\Psi_{12}^\beta &  
=1 +  \frac{1}{a_{12}} b_{213}\cdot Q_{12}^d b_{123},
\end{align}
where $a_{12} = a_{21} \triangleq \| b_{213}\| \| b_{123}\|\in\Re$. Since $\|b_{ijk}\|=\|b_{ij}\times b_{ik}\| = \| R_i^\T s_{ij} \times R_i^T s_{ik}\|=\|s_{ij}\times s_{ik}\|$, the constant $a_{12}$ is fixed according to Assumption \ref{assump:fixeds}, and it is non-zero from Assumption \ref{assump:plane}. Next, we define the configuration error vectors as
\begin{alignat}{2}
e_{12}^\alpha &= (Q^d_{21} b_{21})\times b_{12},&\quad
e_{21}^\alpha &= (Q^d_{12} b_{12})\times b_{21},\label{eqn:ea}\\
e_{12}^\beta &= \frac{1}{a_{12}} (Q^d_{21} b_{213})\times b_{123},&
e_{21}^\beta &= \frac{1}{a_{21}} (Q^d_{12} b_{123})\times b_{213}.\label{eqn:eb}
\end{alignat}
As $b_{12},b_{21}$ are unit vectors, and from the definition of $a_{12},a_{21}$, we can show that $\|e^\alpha_{12}\|,\|e^\alpha_{21}\|,\|e^\beta_{12}\|,\|e^\beta_{21}\|\leq 1$.

We also define the angular velocity errors:
\begin{align}
e_{\Omega_1} = \Omega_1 -\Omega_1^d,\quad
e_{\Omega_2} = \Omega_2 -\Omega_2^d,\label{eqn:eW12}
\end{align}
where the desired absolute angular velocities $\Omega^d_1,\Omega^2_d$ are chosen such that
\begin{align}
\Omega_{12}^d(t) = \Omega_1^d(t) - Q_{21}^d(t) \Omega_2^d(t).\label{eqn:Wd12}
\end{align}
Any desired absolute angular velocities satisfying \refeqn{Wd12} can be chosen. For example, they can be selected as
\begin{align*}
\Omega_{1_d}(t) = \frac{1}{2}\Omega_{12}^d(t),\quad
\Omega_{2_d}(t) = \frac{1}{2}\Omega_{21}^d(t)=-Q^d_{12} \Omega_{1_d}(t).
\end{align*}
Using these desired angular velocities, the derivative of the desired relative attitude can be rewritten as
\begin{align}
\dot Q^d_{12} = Q^d_{12}\hat\Omega^d_1 - \hat\Omega^d_2 Q^d_{12}.
\label{eqn:Qd12dot1}
\end{align}
It is assumed that the desired angular velocities are bounded by known constants.

\begin{assump}\label{assump:Bd}
For known positive constants $B_d$,
\begin{align*}
\|\Omega^d_1(t)\|\leq B^d,\quad \|\Omega^d_2(t)\|\leq B^d,
\end{align*}
for all $t\geq 0$. 
\end{assump}

The properties of these error variables are summarized as follows.

\begin{prop}\label{prop:prop}
For positive constants $k^\alpha_{12}\neq k^\beta_{12}$, define
\begin{align}
\Psi_{12} = k^\alpha_{12} \Psi^\alpha_{12} + k^\beta_{12} \Psi^\beta_{12},\\
e_{12} = k^\alpha_{12} e^\alpha_{12} + k^\beta_{12} e^\beta_{12},\\
e_{21} = k^\alpha_{21} e^\alpha_{21} + k^\beta_{21} e^\beta_{21},
\end{align}
where $k^\alpha_{21}=k^\beta_{21}$, $k^\beta_{21}=k^\beta_{12}$. The following properties hold:
\renewcommand{\labelenumi}{(\roman{enumi})}
\begin{enumerate}\renewcommand{\itemsep}{3pt}
\item $e_{12}=-Q^d_{21}e_{21}$, and $\|e_{12}\|=\|e_{21}\|$.
\item $\frac{d}{dt} \Psi_{12}=e_{12}\cdot e_{\Omega_1}+e_{21}\cdot e_{\Omega_2}$.
\item 
$\|\dot e_{12}\| \leq (k^\alpha_{12}+k^\beta_{12})
(\|e_{\Omega_1}\|+\|e_{\Omega_2}\|)+B^d\|e_{12}\|$,\\
$\|\dot e_{21}\| \leq (k^\alpha_{12}+k^\beta_{12})
(\|e_{\Omega_1}\|+\|e_{\Omega_2}\|)+B^d\|e_{21}\|$.
\item If $\Psi_{12} \leq \psi < 2\text{min}\{k^\alpha_{12},k^\beta_{12}\}$ for a constant $\psi$, then $\Psi$ is quadratic with respect to $\|e_{12}\|$, i.e., the following inequality is satisfied:
\begin{gather}
\underline\psi_{12}\|e_{12}\|^2\leq \Psi_{12} \leq \overline\psi_{12} \|e_{12}\|^2, 
\label{eqn:PsiB}
\end{gather}
where the constants $\underline\psi_{12},\overline\psi_{12}$ are given by
\end{enumerate}
\vspace*{-0.3cm}\begin{align*}
\underline\psi_{12}&=\frac{\text{min}\{k^\alpha_{12},k^\beta_{12}\}}{2\text{max}\{(k^\alpha_{12})^2, (k^\beta_{12})^2,(k^\alpha_{12}-k^\beta_{12})^2\}+2(k^\alpha_{12}+k^\beta_{12})^2},\\
\overline\psi_{12}&=\frac{\text{min}\{k^\alpha_{12},k^\beta_{12}\}(k^\alpha_{12}+k^\beta_{12})}{\text{min}\{(k^\alpha_{12})^2,(k^\beta_{12})^2\}(2\text{min}\{k^\alpha_{12},k^\beta_{12}\}-\psi)}.
\end{align*}

\end{prop}

\begin{proof}
From \refeqn{ea}, \refeqn{RxR}, $e^\alpha_{12}$ is given by
\begin{align*}
e^\alpha_{12} & = (Q^d_{21} b_{21})\times b_{12}= Q^d_{21} \{ b_{21}\times Q^d_{12} b_{12}\} = -Q^d_{21} e^\alpha_{21}.
\end{align*}
Similarly, $e^\beta_{12}=-Q^d_{21} e^\beta_{21}$. These show (i).

Since $x^\T y = \trs{xy^\T}$ for any $x,y\in\Re^3$ and from \refeqn{sbbs}, the configuration error function can be written as
\begin{align}
\Psi_{12} & = k^\alpha_{12}+k^\beta_{12} 
+ k^\alpha_{12}\trs{ b_{21}b_{12}^T Q^d_{21}}
+ \frac{k^\beta_{12}}{a_{12}}\trs{b_{213}b_{123}^T Q^d_{21}}\nonumber\\
& = k^\alpha_{12}+k^\beta_{12} 
+\trs {R_2^T (k_{12}^\alpha s_{21}s_{12}^T +\frac{k^\beta_{12}}{a_{12}}s_{213}s_{123}^T)R_1Q^d_{21}}\nonumber\\
& \triangleq k^\alpha_{12}+k^\beta_{12} 
- \trs{ R_2^T K_{12} R_1 Q^d_{21}},\label{eqn:Psi0}
\end{align}
where $s_{ijk}\in\Re^3$ denotes $s_{ijk}=s_{ij}\times s_{ik}$, and $K_{12}\in\Re^{3\times 3}$. Since $s_{12}=-s_{21}$, $\frac{s_{213}}{\|s_{213}\|}=-\frac{s_{123}}{\|s_{123}\|}$, the matrix $K_{12}$ can be rewritten as
\begin{align}
K_{12}= k^\alpha_{12} s_{12} s_{12}^T +\frac{k^\beta_{12}}{\|s_{123}\|^2} s_{123}s_{123}^T.\label{eqn:K12}
\end{align}

The derivative of the configuration error function with respect to $R_1$ along the direction of $\delta R_1 = R_1\hat \eta_1$ for a vector $\eta_1\in\Re^3$ is given by
\begin{align*}
\D_{R_1}&\Psi_{12}(R_1,R_2,Q^d_{12})\cdot R_1\hat\eta_1 = 
\frac{d}{d\epsilon}\bigg|_{\epsilon=0} \Psi_{12}(R_1\exp\epsilon\hat\eta_1,R_2,Q^d_{12})\\
& = - \trs{ R_2^T K_{12} R_1\hat\eta_1 Q^d_{21}}
=-\tr{\hat\eta_1 Q_{21}^d R_2^\T K_{12} R_1}.
\end{align*}
Using \refeqn{hatxA} and \refeqn{hatxy}, this is rewritten as
\begin{align}
&\D_{R_1}\Psi_{12}(R_1,R_2,Q^d_{12})\cdot R_1\hat\eta_1\nonumber\\
& =\eta_1\cdot (
Q_{21}^d R_2^\T K_{12} R_1-R_1^\T K_{12} R_2 Q_{12}^d)^\vee\label{eqn:e12tmp}\\
& = \eta_1\cdot(-k_{12}^\alpha Q^d_{21}R_2^T s_{12} \times R_1^\T s_{12} 
-\frac{k_{12}^\beta}{a_{12}} Q^d_{21}R_2^T s_{123} \times R_1^\T s_{123}) \nonumber\\
& = \eta_1\cdot(k_{12}^\alpha Q^d_{21}b_{21} \times b_{12}
+\frac{k_{12}^\beta}{a_{12}} Q^d_{21}b_{213} \times R_1^\T b_{123})\nonumber\\
& = \eta_1\cdot e_{12}.\label{eqn:e12}
\end{align}
In short, the error vector $e_{12}$ is the left-trivialized derivative of $\Psi$ with respect to $R_1$. Similarly, we can show that
\begin{align*}
\D_{R_2}&\Psi_{12}(R_1,R_2,Q^d_{12})\cdot R_2\hat\eta_2 = \eta_2\cdot e_{21},\\
\D_{Q^d_{12}}&\Psi_{12}(R_1,R_2,Q^d_{12})\cdot Q^d_{12}\hat\eta_{12} = -\eta_{12}\cdot e_{12}.
\end{align*}
Therefore, the time-derivative of the configuration error function is given by
\begin{align*}
\dot \Psi_{12} & = e_{12}\cdot\Omega_1 +e_{21}\cdot\Omega_2 - e_{12}\cdot \Omega^d_{12}.
\end{align*}
Substituting \refeqn{Wd12} and since $Q^d_{12}e_{12}=-e_{21}$, we obtain
\begin{align*}
\dot \Psi_{12} & = e_{12}\cdot (\Omega_1-\Omega_1^d) + e_{21}\cdot (\Omega_2-\Omega_2^d),
\end{align*}
which shows (ii).

Next we show (iii). From \refeqn{e12tmp}, $e_{12}$ can be written as
\begin{align}
e_{12} = (
Q_{21}^d R_2^\T K_{12} R_1-R_1^\T K_{12} R_2 Q_{12}^d)^\vee.\label{eqn:e12vee}
\end{align}
From attitude kinematic equations \refeqn{Rdot}, \refeqn{Qd12dot1}, the time derivative of the error vector $e_{12}$ is given by
\begin{align*}
\hat{\dot e}_{12} 
& = (Q_{21}^d R_2^\T K_{12} R_1\hat\Omega_1+\hat\Omega_1R_1^\T K_{12} R_2 Q_{12}^d)\\
&\quad-Q_{21}^d (\hat\Omega_2 R_2^\T K_{12} R_1Q^d_{21}+Q^d_{12}R_1^\T K_{12} R_2\hat\Omega_2 )Q_{12}^d\\
&\quad-(\hat\Omega_{1}^dQ_{12}^{d\T} R_2^\T K_{12} R_1
+R_1^\T K_{12} R_2 Q_{12}^d\hat\Omega_{1}^d)\\
&\quad+Q_{21}^d(\hat\Omega_{2}^d R_2^\T K_{12} R_1Q_{21}^{d}
+Q_{21}^{d\T}R_1^\T K_{12} R_2 \hat\Omega_{2}^d)Q^d_{12}.
\end{align*}
Using \refeqn{eW12}, this is rewritten as
\begin{align*}
\hat{\dot e}_{12} 
& = (Q_{21}^d R_2^\T K_{12} R_1\hat e_{\Omega_1}+\hat e_{\Omega_1}R_1^\T K_{12} R_2 Q_{12}^d)\\
&\quad +(Q_{21}^d R_2^\T K_{12} R_1\hat\Omega^d_{1}+\hat\Omega^d_{1}R_1^\T K_{12} R_2 Q_{12}^d)\\
&\quad-Q_{21}^d (\hat e_{\Omega_2} R_2^\T K_{12} R_1Q^d_{21}+Q^d_{12}R_1^\T K_{12} R_2\hat e_{\Omega_2} )Q_{12}^d\\
&\quad-(\hat\Omega_{1}^dQ_{12}^{d\T} R_2^\T K_{12} R_1
+R_1^\T K_{12} R_2 Q_{12}^d\hat\Omega_{1}^d).
\end{align*}
Substituting \refeqn{K12}, and applying \refeqn{hatxy}, \refeqn{STP2}, this reduces to
\begin{align}
\dot e_{12} 
& = k^\alpha_{12}\{-Q_{21}^d b_{21}\times (\hat e_{\Omega_1} b_{12})
+ b_{12}\times Q^d_{21}(\hat e_{\Omega_2} b_{21})\}\nonumber\\
& +\frac{k^\beta_{12}}{a_{12}}\{-Q_{21}^d b_{213}\times (\hat e_{\Omega_1} b_{123})
+ b_{123}\times Q^d_{21}(\hat e_{\Omega_2} b_{213})\}\nonumber\\
& -\hat\Omega^d_{1} e_{12},\label{eqn:e12dot}
\end{align}
which shows the first inequality of (iii) from Assumption \ref{assump:Bd}. The second inequality of (iii) can be shown similarly.

Next, to show (iv), we use the following properties given in~\cite{LeePACC12}. For non-negative constants $f_1,f_2,f_3$, let $F=\text{diag}[f_1,f_2,f_3]\in\Re^{3\times 3}$, and let $P\in\SO$. Define
\begin{gather} 
\Phi=\frac{1}{2}\text{tr}[F(I-P)],\label{eqn:Phi}\\
e_P=\frac{1}{2}(FP-P^{\T}F)^\vee,\label{eqn:eP}
\end{gather}
Then, $\Phi$ is bounded by the square of the norm of $e_P$ as
\begin{gather}
\frac{h_1}{h_2+h_3}\|e_P\|^2\leq \Phi \leq\frac{h_1h_4}{h_5(h_1-\phi)}\|e_P\|^2, 
\label{eqn:PhiB}
\end{gather}
if $\Phi<\phi<h_1$ for a constant $\phi$, where $h_i$ are given by
\begin{align*} 
h_1 &= \text{min}\{f_1+f_2,~ f_2+f_3,~ f_3+f_1\}, \\
h_2 &= \text{max}\{(f_1-f_2)^2,~ (f_2-f_3)^2,~ (f_3-f_1)^2\}, \\
h_3 &= \text{max}\{(f_1+f_2)^2,~ (f_2+f_3)^2,~ (f_3+f_1)^2\}, \\
h_4 &= \text{max}\{f_1+f_2,~ f_2+f_3,~ f_3+f_1\}, \\
h_5 &= \text{min}\{(f_1+f_2)^2,~(f_2+f_3)^2,~(f_3+f_1)^2\}.
\end{align*}

We apply this property by showing that the configuration error function $\Psi_{12}$ given at \refeqn{Psi0} can be rewritten as \refeqn{Phi}. At \refeqn{K12}, the matrix $K_{12}$ can be decomposed into $K_{12}=UGU^T$, where $G$ is the diagonal matrix given by $G=\mathrm{diag}[k^\alpha_{12},k^\beta_{12},0]\in\Re^3$, and $U$ is an orthonormal matrix defined as $U=[s_{12}, \frac{s_{123}}{\|s_{123}\|}, \frac{s_{12}\times s_{123}}{\|s_{12}\times s_{123}\|} ]\in\SO$. Using the property, $\trs{AB}=\trs{BA}=\trs{B^\T A^T}=\trs{A^\T B^\T}$, the configuration error function can be rearranged as
\begin{align}
\Psi_{12} 
& =\trs{G(I- U^\T R_1 Q^d_{21} R_2^\T U)}.\label{eqn:PsiG}
\end{align}
Therefore, if we choose $F=2G$ and $P=U^TR_1Q^d_{21}R_2^T U$, we obtain $\Psi_{12}=\Phi$. Also substituting these into \refeqn{eP},
\begin{align*}
\hat e_P & = GU^TR_1Q^d_{21}R_2^T U- U^T R_2 Q^d_{12}R_1^T U G\\
&=U^T (K_{12} R_1Q^d_{21}R_2^T -  R_2 Q^d_{12}R_1^T K_{12}) U\\
&=U^T R_2 Q^d_{12}(Q^d_{21}R_2^TK_{12} R_1 -  R_1^T K_{12}R_2Q^d_{12})Q^d_{21}R_2^T U. 
\end{align*}
From this, we have $e_P=U^T R_2 Q^d_{12} e_{12}$ by \refeqn{e12vee}, \refeqn{RxR}. Therefore, $\|e_P\|=\|e_{12}\|$. Then, \refeqn{PhiB} yields (iv) with $f_1=2k^\alpha_{12}$, $f_2=2k^\beta_{12}$, $f_3=0$.

\end{proof}

Using these properties, we develop a control system to track the given desired relative attitude as follows.

\begin{prop}\label{prop:2}
Consider the attitude dynamics of spacecraft given by \refeqn{Wdot}, \refeqn{Rdot} for $i\in\{1,2\}$, with the LOS measurements specified at \refeqn{Ap}. A desired relative attitude trajectory is given by \refeqn{dotQd12}. For positive constants $k^\alpha_{12}\neq k^\beta_{12},k^\alpha_{21}=k^\alpha_{12},k^\beta_{21}=k^\beta_{12},k_{\Omega_{1}},k_{\Omega_{2}}$, control inputs are chosen as
\begin{align}
u_i = -e_{ij} -k_{\Omega_i}e_{\Omega_{i}} +\hat\Omega^d_i J_i(e_{\Omega_i}+\Omega^d_i)+J\dot\Omega^d_i,\label{eqn:ui}
\end{align}
where $(i,j)\in\mathcal{E}$. Then, the following properties hold:
\renewcommand{\labelenumi}{(\roman{enumi})}
\begin{enumerate}\renewcommand{\itemsep}{3pt}
\item There are four types of equilibrium, given by the desired equilibrium $(Q,\Omega_{12})=(Q^d_{12},\Omega^d_{12})$, and the relative configurations represented by $Q^d_{12}=R_2^T U D U^T R_1$ and $\Omega_{12}=\Omega_{12}^d$ where $D\in\{\mathrm{diag}[1,-1,-1],\mathrm{diag}[-1,1,-1],\mathrm{diag}[-1,-1,1]\}$ and $U\in\SO$ is the matrix composed of eigenvectors of $K_{12}$ given at \refeqn{K12}.
\item The desired equilibrium is almost globally exponentially stable, and a (conservative) estimate to the region of attraction is given by
\begin{gather}
\Psi_{12}(0) \leq \psi < 2\text{min}\{k^\alpha_{12},k^\beta_{12}\},\label{eqn:ROA0}\\
\sum_{i=1,2} \lambda_{M_i} \|e_{\Omega_i}(0)\|^2 \leq 2(\psi-\Psi_{12}(0)),\label{eqn:ROA1}
\end{gather}
where $\psi$ is a positive constant satisfying $\psi< 2\min\{k^\alpha_{12},k^\beta_{12}\}$, and $\lambda_{M_i}$ denotes the maximum eigenvalue of $J_i$. 
\item The undesired equilibria are unstable.
\end{enumerate}
\end{prop}

\begin{proof}
From \refeqn{Wdot}, \refeqn{ui}, and rearranging, the time-derivative of $J_ie_{\Omega_i}$ is given by
\begin{align}
J_i\dot e_{\Omega_i} & = -(e_{\Omega_i}+\Omega^d_i)\times J_i(e_{\Omega_i}+\Omega^d_i) + u_i -J_i\dot\Omega^d_i,\label{eqn:Jeidot0}\\
& =(J_ie_{\Omega_i} +J\Omega^d_i)\hat e_{\Omega_i} - e_{ij} - k_{\Omega_i} e_{\Omega_i}.\label{eqn:Jeidot}
\end{align}
The equilibrium configurations are where $e_{12}=e_{21}=e_{\Omega_{1}}=e_{\Omega_1}=0$, which corresponds to the critical points of the configuration error function given by \refeqn{PsiG}. In~\cite{BulLew05}, it has been shown that there are four critical points:
\begin{align*}
R_1 Q^d_{21}R_2^\T \in \{I, UD_1U^\T, UD_2U^\T, UD_3U^\T\},
\end{align*}
where $D_1=\mathrm{diag}[1,-1,-1]$, $D_2=\mathrm{diag}[-1,1,-1]$, $D_3=\mathrm{diag}[-1,-1,1]$. This shows (i).

Next, we show exponential stability of the desired equilibrium. A sufficient condition on the initial conditions to satisfy \refeqn{PsiB} is obtained from the following variable:
\begin{align*}
\mathcal{U} = \frac{1}{2} e_{\Omega_1}\cdot J_1e_{\Omega_1} 
+\frac{1}{2} e_{\Omega_2}\cdot J_2e_{\Omega_2}+ \Psi_{12}.
\end{align*}
From \refeqn{Jeidot} and the property (ii) of Proposition \ref{prop:prop}, the time-derivative of $\mathcal{U}$ is simply given by
\begin{align}
\dot{\mathcal{U}} 
& = -k_{\Omega_1} \|e_{\Omega_1}\|^2 -k_{\Omega_2} \|e_{\Omega_2}\|^2\leq 0,\label{eqn:Udot}
\end{align}
which implies that $\mathcal{U}(t)$ is non-increasing. For the initial conditions satisfying \refeqn{ROA0}, \refeqn{ROA1}, we have
\begin{align*}
\mathcal{U}(0) \leq \frac{1}{2}\sum_{i=1,2}\lambda_{M_i} \|e_{\Omega_i}(0)\|^2 + \Psi_{12}(0) \leq \psi.
\end{align*}
As $\mathcal{U}(0)$ is non-increasing, 
\begin{align*}
\Psi_{12}(t)\leq \mathcal{U}(t) \leq \mathcal{U} (0) \leq \psi < 2\min\{k^\alpha_{12},k^\beta_{12}\}.
\end{align*}
Therefore, $\Psi_{12}(t)\leq\psi<h_1$ for all $t\geq 0$, and the inequality \refeqn{PsiB} is satisfied. 

Let a Lyapunov function be 
\begin{align*}
\mathcal{V} = \mathcal{U} + c(e_{12}\cdot e_{\Omega_1} + e_{21}\cdot e_{\Omega_2}),
\end{align*}
for a constant $c$. Using \refeqn{PsiB}, it can be shown that
\begin{align}
\sum_{i,j\in\mathcal{E}} z_{ij}^T \underline M_{ij} z_{ij}\leq \mathcal{V}\leq
\sum_{i,j\in\mathcal{E}} z_{ij}^T \overline M_{ij} z_{ij},
\label{eqn:VB}
\end{align}
where $z_{ij}=[\|e_{ij}\|,\|e_{\Omega_i}\|]\in\Re^2$, and the matrices $\underline M_{ij}, \overline M_{ij}\in\Re^{2\times 2}$ are defined as
\begin{align}
\underline M_{ij} = \frac{1}{2}\begin{bmatrix}
\underline \psi_{ij} & -c \\
-c & \lambda_{m_i}
\end{bmatrix},\quad
\overline M_{ij} = \frac{1}{2}\begin{bmatrix}
\overline \psi_{ij} & c \\
c & \lambda_{M_i}
\end{bmatrix},\label{eqn:Bij}
\end{align}
for $(i,j)\in\mathcal{E}=\{(1,2),(2,1)\}$, where it is assumed that $\underline\psi_{ij}=\underline\psi_{ji}$, $\overline\psi_{ij}=\overline\psi_{ji}$. From \refeqn{Udot}, we obtain
\begin{align}
\dot{\mathcal{V}} = \sum_{i,j\in\mathcal{E}} -k_{\Omega_i} \|e_{\Omega_i}\|^2  + c( \dot e_{ij}\cdot e_{\Omega_i} + e_{ij}\cdot \dot e_{\Omega_i}).\label{eqn:Vdot0}
\end{align}
From \refeqn{Jeidot}, and using the fact that $\|e_{ij}\|\leq k^\alpha_{ij}+k^\beta_{ij}$, we have
\begin{align}
\dot e_{\Omega_i}\cdot e_{ij} 
&\leq - \frac{1}{\lambda_{M_i}}\|e_{ij}\|^2
+\frac{\lambda_{M_i}}{\lambda_{m_i}} (k^\alpha_{ij}+k^\beta_{ij})\|e_{\Omega_i}\|^2\nonumber\\
&\quad +\frac{\lambda_{M_i}B^d+ k_{\Omega_i}}{\lambda_{m_i}}\|e_{ij}\|\|e_{\Omega_i}\|.
\label{eqn:eWidoteij}
\end{align}
Together with the property (iii) of Proposition \ref{prop:prop}, this yields the following inequality of $\dot{\mathcal{V}}$:
\begin{align*}
\dot{\mathcal{V}} &= \sum_{i,j\in\mathcal{E}}
 -(k_{\Omega_i}-c(k^\alpha_{ij}+k^\beta_{ij})(1+\frac{\lambda_{M_i}}{\lambda_{m_i}})) \|e_{\Omega_i}\|^2 \\
&\quad - \frac{c}{\lambda_{M_i}}\|e_{ij}\|^2
+c(k^\alpha_{ij}+k^\beta_{ij})\|e_{\Omega_i}\|\|e_{\Omega_j}\|\\
&\quad+\frac{c}{\lambda_{m_i}}((\lambda_{M_i}+\lambda_{m_i})B^d+k_{\Omega_i}) \|e_{ij}\|\|e_{\Omega_i}\|.
\end{align*}
This can be rewritten as the following matrix form:
\begin{align}
\dot{\mathcal{V}} \leq -z_{12}^T W_{12} z_{12} - z_{21}^T W_{21} z_{21}-\zeta_{12}^T Y_{12} \zeta_{12},\label{eqn:VdotB}
\end{align}
where $\zeta_{12}=[\|e_{\Omega_1}\|,\|e_{\Omega_2}\|]^T\in\Re^2$ and the matrices $W_{12},W_{21},Y_{12}\in\Re^{2\times 2}$ are given by
\begin{align}
W_{ij}=\frac{1}{2}\begin{bmatrix}
\frac{c}{\lambda_{M_i}}
& -\frac{c}{\lambda_{m_i}}(\bar\lambda_i B^d+k_{\Omega_i})\\
-\frac{c}{\lambda_{m_i}}(\bar\lambda_iB^d+k_{\Omega_i})
& k_{\Omega_i}-c\bar k_{ij}(1+\frac{\lambda_{M_i}}{\lambda_{m_i}})
\end{bmatrix},\label{eqn:Wij}\\
Y_{12}= \frac{1}{2}\begin{bmatrix}
k_{\Omega_1}-c\bar k_{12}(1+\frac{\lambda_{M_1}}{\lambda_{m_1}}) &
-2c\bar k_{12} \\
-2c\bar k_{12} 
& k_{\Omega_2}-c\bar k_{12}(1+\frac{\lambda_{M_2}}{\lambda_{m_2}}) 
\end{bmatrix}.\label{eqn:Y12}
\end{align}
for $(i,j)\in\mathcal{E}$. Here, $\bar\lambda_i$ denotes $\bar\lambda_i=\lambda_{M_i}+\lambda_{m_i}$, and $\bar k_{ij}$ denotes $\bar k_{ij}=k^\alpha_{ij}+k^\beta_{ij}$. It can be shown that if the constant $c$ is sufficiently small, then all of the matrices $\underline{M}_{ij},\overline M_{ij}, W_{ij}$ for $(i,j)\in\mathcal{E}$, and $Y_{12}$ at \refeqn{VB} and \refeqn{VdotB} are positive definite. For example, $\underline{M}_{ij}$ is positive definite if $c<\sqrt{\underline{\psi}_{ij}\lambda_{m_i}}$. This shows that $\mathcal{V}$ is positive definite and decrescent, and $\dot{\mathcal{V}}$ is negative definite. Therefore, the desired equilibrium is exponentially stable. 

Next, we show (iii). At the first type of undesired equilibria given by $R_1Q^d_{21}R_2^T = UD_1U^\T$ and $e_{\Omega_1}=e_{\Omega_2}=0$, the value of the Lyapunov function becomes $\mathcal{V}= 2k^\beta_{12}$. Define
\begin{align*}
\mathcal{W} = 2k^\beta_{12} -\mathcal{V}. 
\end{align*}
Then, $\mathcal{W}=0$ at the undesired equilibrium, and we have
\begin{align*}
-\sum_{(i,j)\in\mathcal{E}}\{\frac{\lambda_{M_i}}{2}\|e_{\Omega_i}\|^2 +c\|e_{ij}\|\|e_{\Omega_i}\|\} + (2k^\beta_{12} - \Psi) \leq \mathcal{W}.
\end{align*}
Due to the continuity of $\Psi$, we can choose $R_1$ and $R_2$ arbitrary close to the undesired equilibrium such that $(2k^\beta_2 - \Psi)>0$. Therefore, if $\|e_{\Omega_i}\|$ is sufficiently small, we obtain $\mathcal{W}>0$. Therefore, at any arbitrarily small neighborhood of the undesired equilibrium, there exists a set in which $\mathcal{W}>0$, and we have $\dot{\mathcal{W}}=-\dot{\mathcal{V}} >0$ from \refeqn{VdotB}. Therefore, the undesired equilibrium is unstable~\cite[Theorem 3.3]{Kha96}. The instability of other types of equilibrium can be shown similarly. This shows (iii).

The region of attraction to the desired equilibrium excludes the union of stable manifolds to the unstable equilibria. But, the union of stable manifolds has less dimension than the tangent bundle of the configuration manifold. Therefore, the measure of the stable manifolds to the unstable equilibria is zero. This implies the desired equilibrium is almost globally exponentially stable~\cite{ChaSanICSM11}, which shows (ii).
\end{proof}

This states that almost all solutions of the proposed control system, excluding a class of solutions starting from a specific set that has a zero-measure, asymptotically track given the
desired relative attitude. As the control inputs are expressed in terms of observations, in addition to angular velocities, and the full relative attitude does not have to be constructed at each time. These results can be considered as a generalization of the preliminary work in \cite{LeePACC12}, but it is a nontrivial extension as the several properties of the error variables should be considered to show a stronger exponential stability for tracking problems.

\section{Relative Attitude Formation Tracking}

The relative attitude control system between two spacecraft developed in the previous section can be used as a building block for a relative attitude formation control system for multiple spacecraft. In this section, we generalize it for daisy-chained relative attitude formation control network.

\subsection{Relative Attitude Tracking Between Three Spacecraft}

We first consider relative attitude formation tracking between three spacecraft, given by Spacecraft 1, 2, and 3, illustrated at Figure \ref{fig:FS}. The corresponding edge set and the assignment set used in this subsection are given by
\begin{gather}
\mathcal{E} =\{(1,2),(2,1),(2,3),(3,2)\},\\
\mathcal{A} =\{(1,2,3),(2,1,3),(2,3,1),(3,2,1)\}.\label{eqn:App}
\end{gather}
For given relative attitude commands, $Q^d_{12}(t),Q^d_{23}(t)$, the goal is to design control inputs such that $Q_{12}(t)\rightarrow Q^d_{12}(t)$ and $Q_{23}(t)\rightarrow Q^d_{23}(t)$ as $t\rightarrow \infty$.

The definition of error variables and their properties developed in the previous section for two spacecraft are readily generalized to any $(i,j,k)\in\mathcal{A}$ in this section. For example, the kinematic equation for the desired relative attitude $Q^d_{23}$ is obtained from \refeqn{dotQd12} as
\begin{align*}
\dot Q^d_{23} = Q^d_{23}\hat\Omega^d_{23},
\end{align*}
where $\Omega^d_{23}$ is the desired relative angular velocity. Other configuration error functions and error vectors between Spacecraft 2 and Spacecraft 3 are defined similarly. 

The desired absolute angular velocities for each spacecraft, namely $\Omega^d_1$, $\Omega^d_2$, and $\Omega^d_{3}$ should be properly defined. For the given $\Omega^d_{12}$, $\Omega^d_{23}$, they can be arbitrarily chosen such that
\begin{align}
\Omega_{12}^d(t) = \Omega_1^d(t) - Q_{21}^d(t) \Omega_2^d(t),\\
\Omega_{23}^d(t) = \Omega_2^d(t) - Q_{32}^d(t) \Omega_3^d(t).
\end{align}
For example, they can be chosen as $\Omega^d_1= \Omega^d_{12}$, $\Omega^d_2= 0$, $\Omega^d_3 = -Q^d_{23}\Omega^d_{23}$. Assumption \ref{assump:Bd} is considered to be satisfied such that each of the desired angular velocity is bounded by a known constant $B^d$.

\begin{prop}\label{prop:RAT3}
Consider the attitude dynamics of spacecraft given by \refeqn{Wdot}, \refeqn{Rdot} for $i\in\{1,2,3\}$, with the LOS measurements specified at \refeqn{App}. Desired relative attitudes are given by $Q^d_{12}(t)$, $Q^d_{23}(t)$. For positive constants $k^\alpha_{ij},k^\beta_{ij},k_{\Omega_{i}}$ with $k^\alpha_{ij}\neq k^\beta_{ij}$, $k^\alpha_{ij}=k^\alpha_{ji}$, $k^\beta_{ij}=k^\beta_{ji}$ for $(i,j)\in\mathcal{E}$, 
\begin{align}
u_1 &= -e_{12} -k_{\Omega_1}e_{\Omega_{1}} +\hat\Omega^d_1 J_1(e_{\Omega_1}+\Omega^d_1)+J\dot\Omega^d_1,\label{eqn:u1}\\
u_2 &= -\frac{1}{2}(e_{21}+e_{23}) -k_{\Omega_2}e_{\Omega_{2}} +\hat\Omega^d_2 J_2(e_{\Omega_2}+\Omega^d_2)+J\dot\Omega^d_2,\label{eqn:u2}\\
u_3 &= -e_{32} -k_{\Omega_3}e_{\Omega_{3}} +\hat\Omega^d_3 J_3(e_{\Omega_3}+\Omega^d_3)+J\dot\Omega^d_3\label{eqn:u1},
\end{align}
Then, the desired relative attitude configuration is almost globally exponentially stable, and a (conservative) estimate to the region of attraction is given by
\begin{gather}
\Psi_{12}(0)+\Psi_{23}(0)
 \leq \psi < 2\min\{k^\alpha_{12},k^\beta_{12},k^\alpha_{23},k^\beta_{23}\},\label{eqn:ROA2}\\
\begin{aligned}
 \lambda_{M_1} \|e_{\Omega_1}(0)\|^2
&+2\lambda_{M_2} \|e_{\Omega_2}(0)\|^2
+\lambda_{M_3} \|e_{\Omega_3}(0)\|^2\\
& \leq 2(\psi-\Psi_{12}(0)-\Psi_{23}(0)),
\end{aligned}\label{eqn:ROA3}
\end{gather}
where $\psi$ is a positive constant satisfying $\psi< 2\min\{k^\alpha_{12},k^\beta_{12},k^\alpha_{23},k^\beta_{23}\}$.

\end{prop}

\begin{proof}
The time-derivative of $J_1e_1$ and $J_3e_3$ are given by \refeqn{Jeidot}, and the time-derivative of $J_2e_2$ is given by
\begin{align}
J_2\dot e_{\Omega_2} 
& =(J_2e_{\Omega_2} +J\Omega^d_2)\hat e_{\Omega_2} - \frac{1}{2}(e_{21}+e_{23}) - k_{\Omega_2} e_{\Omega_2}.\label{eqn:Je2dot}
\end{align}
Define 
\begin{align*}
\mathcal{U} & = \frac{1}{2}e_{\Omega_1}\cdot J_1 e_{\Omega_1}
+e_{\Omega_2}\cdot J_2 e_{\Omega_2}
+\frac{1}{2}e_{\Omega_3}\cdot J_3 e_{\Omega_3}\nonumber\\
&\quad +\Psi_{12}+\Psi_{23}.
\end{align*}
From \refeqn{Jeidot}, \refeqn{Je2dot}, we have
\begin{align*}
\dot{\mathcal{U}} = -k_{\Omega_1}\|e_{\Omega_1}\|^2
-2k_{\Omega_3}\|e_{\Omega_2}\|^2
-k_{\Omega_3}\|e_{\Omega_3}\|^2,
\end{align*}
which implies that $\mathcal{U}(t)$ is non-increasing. For the initial conditions satisfying \refeqn{ROA2} and \refeqn{ROA3}, we have $\mathcal{U}(0)\leq \psi$. Therefore,
\begin{align*}
\Psi_{12}(t)+\Psi_{23}(t)\leq \mathcal{U}(t) \leq \mathcal{U}(0) \leq \psi < 2\min\{k^\alpha_{12},k^\beta_{12},k^\alpha_{23},k^\beta_{23}\}.
\end{align*}
Therefore, the inequality \refeqn{PsiB} holds for both of $\Psi_{12}$ and $\Psi_{23}$. 

Let a Lyapunov function be
\begin{align}
\mathcal{V} = \mathcal{U} + c e_{\Omega_1}\cdot e_{12} + ce_{\Omega_2}\cdot (e_{21}+e_{23})
+ce_{\Omega_3}\cdot e_{32}.
\end{align}
From \refeqn{PsiB}, we can show that this Lyapunov function satisfies the inequality given by \refeqn{VB}. The time-derivative of the Lyapunov function is given by \refeqn{Vdot0}. 

For $(i,j)\in\{(1,2),(3,2)\}$, the upper bound of $\dot e_{\Omega_{i}}\cdot e_{ij}$ is given by \refeqn{eWidoteij}. From \refeqn{Je2dot}, the upper bound of $\dot e_{\Omega_{2}}\cdot (e_{21}+e_{23})$ is given by
\begin{align*}
&\dot e_{\Omega_2}\cdot (e_{21}+e_{23})
\leq - \frac{1}{\lambda_{M_2}}\|e_{21}+e_{23}\|^2\nonumber\\
&+\frac{\lambda_{M_2}}{\lambda_{m_2}} (\overline k_{21}+\overline k_{23})\|e_{\Omega_2}\|^2
 +\frac{\lambda_{M_2}B^d+ k_{\Omega_2}}{\lambda_{m_2}}\|e_{21}+e_{23}\|\|e_{\Omega_2}\|.
\end{align*}
The upper bounds of $\|\dot e_{32}\|$ and $\|\dot e_{12}\|$ are given by the property (iii) of Proposition \ref{prop:prop}. Additionally, using \refeqn{e12dot}, we can show that
\begin{align*}
\|\dot e_{21}+\dot e_{23}\|& \leq (\overline k_{21}+\overline k_{23})
(\|e_{\Omega_1}\|+2\|e_{\Omega_2}\|+\|e_{\Omega_3}\|)\nonumber\\
&\quad +B^d\| e_{21}+e_{23}\|.
\end{align*}
Applying these bounds to the expression of $\dot{\mathcal{V}}$ and rearranging, we obtain
This can be written as a matrix form as
\begin{align}
\dot{\mathcal{V}} & \leq -z_{12}^T W_{12} z_{12}
- z_{213}^T W_{213} z_{213} - z_{32}^T W_{32} z_{32}
\nonumber\\
&\quad
 -\zeta_{21}^T Z_{21} \zeta_{21}
- \zeta_{23}^T Z_{23} \zeta_{23},\label{eqn:Vdot2}
\end{align}
where the matrix $W_{12},W_{23}\in\Re^{2\times 2}$ are given as \refeqn{Wij}, and $z_{213}=[\|e_{21}+e_{23}\|, \|e_{\Omega_2}\|]\in\Re^2$. The matrices $W_{213}$, $Z_{21}$, $Z_{23}\in\Re^{2\times 2}$ are defined as
\begin{align*}
W_{213}&=\frac{1}{2}\begin{bmatrix}
\frac{c}{\lambda_{M_2}}
& -\frac{c}{\lambda_{m_2}}(\bar\lambda_2 B^d+k_{\Omega_2})\\
-\frac{c}{\lambda_{m_2}}(\bar\lambda_2B^d+k_{\Omega_2})
& 2k_{\Omega_2}-c\bar k_{213}(2+\frac{\lambda_{M_2}}{\lambda_{m_2}}) 
\end{bmatrix},\\
Z_{21}&= \frac{1}{4}\begin{bmatrix}
2k_{\Omega_2}-c\bar k_{213}(2+\frac{\lambda_{M_2}}{\lambda_{m_2}}) 
& -c(6\bar k_{12}+2\bar k_{23}) \\
-c(6\bar k_{12}+2\bar k_{23})
& k_{\Omega_1}-c\bar k_{13}(1+\frac{\lambda_{M_1}}{\lambda_{m_1}})
\end{bmatrix},\\
Z_{23}&= \frac{1}{4}\begin{bmatrix}
2k_{\Omega_2}-c\bar k_{213}(2+\frac{\lambda_{M_2}}{\lambda_{m_2}}) 
& -c(6\bar k_{23}+2\bar k_{12}) \\
-c(6\bar k_{23}+2\bar k_{12})
& k_{\Omega_3}-c\bar k_{13}(1+\frac{\lambda_{M_3}}{\lambda_{m_3}})
\end{bmatrix},
\end{align*}
where $\bar k_{213}=\bar k_{12}+\bar k_{23}$. It can be shown that if the constant $c$ is sufficient small, all of matrices at \refeqn{VB} and \refeqn{Vdot2} are positive definite, which implies 
\begin{align*}
\dot{\mathcal{V}} & \leq -\lambda_{m}(W_{12}) (\|e_{12}\|^2+\|e_{\Omega_1}\|^2)\\
&\quad -\lambda_{m}(W_{213})(\|e_{21}+e_{23}\|^2+\|e_{\Omega_2}\|^2)\\
&\quad -\lambda_{m}(W_{32})(\|e_{31}\|^2+\|e_{\Omega_3}\|^2)\\
&\leq -\lambda_{m}(W_{12}) \braces{\frac{1}{2}(\|e_{12}\|^2+\|e_{21}\|^2)+\|e_{\Omega_1}\|^2}\\
&\quad-\lambda_{m}(W_{213})\|e_{\Omega_2}\|^2\\
&\quad -\lambda_{m}(W_{32}) \braces{\frac{1}{2}(\|e_{32}\|^2+\|e_{23}\|^2)+\|e_{\Omega_3}\|^2},
\end{align*}
where $\lambda_m(\cdot)$ denotes the minimum eigenvalue of a matrix, and we use the fact that $\|e_{12}\|=\|e_{21}\|$, $\|e_{23}\|=\|e_{32}\|$. Therefore, the desired equilibrium is exponentially stable. 

To show \textit{almost} exponential stability, it is required that the fifteen types of the undesired equilibria, corresponding to the critical points of $\Psi_{12}$ and $\Psi_{23}$, are unstable. This is similar to the proof of the property (iii) of Proposition \ref{prop:2}, and it is omitted.
\end{proof}

The control inputs for Spacecraft 1 and Spacecraft 3 at the both ends of graph are identical to \refeqn{ui} at Proposition \ref{prop:2}. The control input for Spacecraft 2, which are paired with both of Spacecraft 1 and 3, is also similar to \refeqn{ui} except that the configuration error vectors for Spacecraft 2, namely $e_{21}$ and $e_{23}$ are averaged. These ideas can be generalized to relative attitude formation tracking between an arbitrary number of spacecraft as follows.

\subsection{Relative Attitude Formation Tracking Between $n$ Spacecraft}

Consider a formation of $n$ spacecraft, i.e., $\mathcal{N}=\{1,\ldots,n\}$. According to Assumption \ref{assump:chain}, spacecraft are paired serially in the edge set. For convenience, it is assumed that spacecraft are numbered such that the edge set is given by
\begin{align}
\mathcal{E}=\{(1,2),\ldots,(n-1,n),(2,1),\ldots (n,n-1)\}.\label{eqn:En}
\end{align}
The assignment set is given by \refeqn{A} for an arbitrary assignment map satisfying Assumption \ref{assump:plane}. The desired relative attitudes $Q^d_{ij}$ for $(i,j)\in\mathcal{E}$ are prescribed. The definition of error variables and their properties developed in Section \ref{sec:RAT2} are generalized to any $(i,j,k)\in\mathcal{A}$. The desired absolute angular velocities $\Omega^d_i$ for $i\in\mathcal{N}$ are chosen such that 
\begin{align}
\Omega^d_{ij}(t) = \Omega^d_i(t) - Q^d_{ji}(t) \Omega^d_{j}(t)\quad\text{for $(i,j)\in\mathcal{E}$}. \label{eqn:Wid}
\end{align}

\begin{prop}
Consider the attitude dynamics of spacecraft given by \refeqn{Wdot}, \refeqn{Rdot} for $i\in\{1,\ldots,n\}$, with the LOS measurements specified by \refeqn{En},\refeqn{A}. Desired relative attitudes are given by $Q^d_{ij}(t)$ for $(i,j)\in\mathcal{E}$. For positive constants $k^\alpha_{ij},k^\beta_{ij},k_{\Omega_{i}}$ with $k^\alpha_{ij}\neq k^\beta_{ij}$, $k^\alpha_{ij}=k^\alpha_{ji}$, $k^\beta_{ij}=k^\beta_{ji}$, the control inputs chosen as
\begin{align}
u_1 &= -e_{12} -k_{\Omega_1}e_{\Omega_{1}} +\hat\Omega^d_1 J_1(e_{\Omega_1}+\Omega^d_1)+J\dot\Omega^d_1,\label{eqn:u1n}\\
u_p &= -\frac{1}{2}(e_{p,p-1}+e_{p,p+1}) -k_{\Omega_p}e_{\Omega_{p}} +\hat\Omega^d_p J_p(e_{\Omega_p}+\Omega^d_p)\nonumber\\
&\quad +J\dot\Omega^d_p\quad\text{for $p\in\{2,\ldots,n-1\}$,}\label{eqn:up}\\
u_n &= -e_{n,n-1} -k_{\Omega_n}e_{\Omega_{n}} +\hat\Omega^d_n J_n(e_{\Omega_n}+\Omega^d_n)+J\dot\Omega^d_n\label{eqn:un},
\end{align}
Then, the desired relative attitude configuration is almost globally exponentially stable, and a (conservative) estimate to the region of attraction is given by
\begin{gather*}
\sum_{i=1}^{n-1}\Psi_{i,i+1}(0)
 \leq \psi < 2\min_{1\leq i\leq n-1}\{k^\alpha_{i,i+1},k^\beta_{i,i+1}\},\\
\begin{split}
 \lambda_{M_1} \|e_{\Omega_1}(0)\|^2
&+2\sum_{i=2}^{n-1}\lambda_{M_i} \|e_{\Omega_i}(0)\|^2
+\lambda_{M_n} \|e_{\Omega_n}(0)\|^2\\
& \leq 2(\psi-\sum_{i=1}^{n-1}\Psi_{i,i+1}(0)),
\end{split}
\end{gather*}
where $\psi$ is a positive constant satisfying $\psi< 2\min_{1\leq i\leq n-1}\{k^\alpha_{i,i+1},k^\beta_{i,i+1}\}$.
\end{prop}
The proof of this proposition is a straightforward, but tedious extension of the proof of Proposition \ref{prop:RAT3}. Due to page limit, the detailed proof is omitted, but numerical results for multiple spacecraft are provided in the next section.

\section{Numerical Example}

\setlength{\unitlength}{0.1\columnwidth}
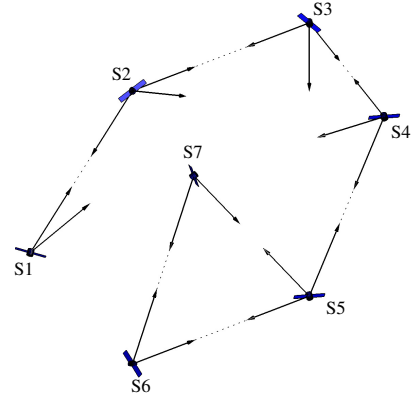
\begin{figure}
\scriptsize\selectfont
\centerline{
\begin{picture}(6,6)(0,0)
\put(0,0){\includegraphics[width=0.60\columnwidth]{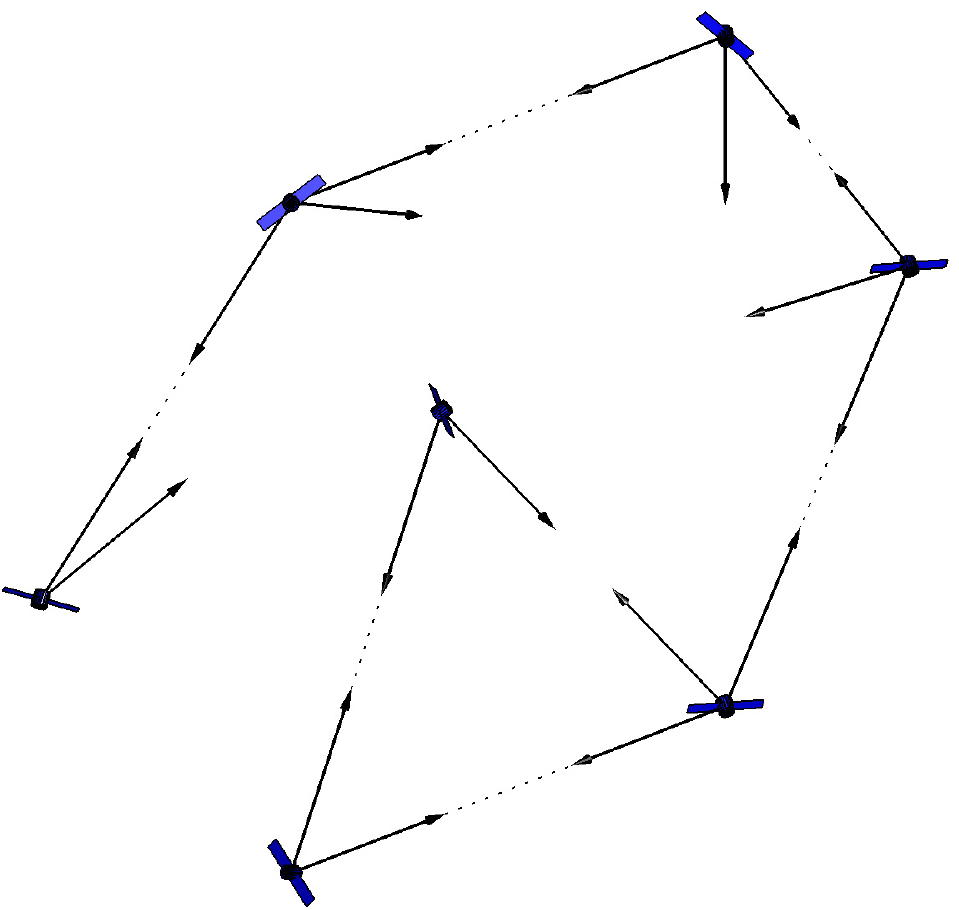}}
\put(0,1.6){S1}
\put(1.5,4.6){S2}
\put(4.6,5.6){S3}
\put(5.8,3.7){S4}
\put(4.8,1.0){S5}
\put(1.8,-0.2){S6}
\put(2.6,3.4){S7}
\end{picture}}
\vspace*{0.2cm}
\caption{Relative attitude formation tracking for 7 spacecraft: the lines-of-sight measured by each spacecraft are denoted by arrows, and the dotted line between two spacecraft implies that they are paired at the edge set. }\label{fig:SFCN}
\end{figure}

Consider the formation of seven spacecraft illustrated at Figure \ref{fig:SFCN}. The corresponding edge is given by \refeqn{En} with $n=7$, and the assignment set is
\begin{align*}
\mathcal{A}=\{&(1,2,3),(2,1,3),\,(2,3,4),(3,2,4),\,(3,4,5),(4,3,5),\\
&(4,5,7),(5,4,7),\, (5,6,7),(6,5,7),\, (6,7,5),(7,6,5)\}.
\end{align*}
The desired relative attitudes for $Q^d_{34}$ and $Q^d_{45}$ are given in terms of 3-2-1 Euler angles as $Q^d_{34}(t)=Q^d_{34}(\alpha(t),\beta(t),\gamma(t))$, $Q^d_{45}(t)=Q^d_{45}(\phi(t),\theta(t),\psi(t))$, where
\begin{gather*}
\alpha(t)=\sin 0.5t,\;\beta(t)=0.1,\;\gamma(t)=\cos t,\\
\phi(t)=0,\;\theta(t)=-0.1+\cos 0.2t,\;\psi(t)=0.5\sin 2t,
\end{gather*}
and $Q_{12}^d(t)=Q^d_{23}(t)=Q_{56}^d(t)=I$, $Q^d_{67}(t)=(Q^d_{45}(t))^\T$. It is chosen that $\Omega^d_4(t)=0$, and other desired absolute angular velocities are selected to satisfy \refeqn{Wid}.

The initial attitudes for Spacecraft 3 and 6 are chosen as $R_3(0)=\exp (0.999\pi\hat e_1)$ and $R_6(0)=\exp (0.990\pi\hat e_2)$, where $e_1=[1,0,0]^\T,e_2=[0,1,0]^\T\in\Re^3$. The initial attitudes for other spacecraft are chosen as the identity matrix. The resulting initial errors for the relative attitudes $Q_{23}$ and $Q_{67}$ are $0.99\pi\,\mathrm{rad}=179.82^\circ$. The initial angular velocity is chosen as zero for every spacecraft.

The inertia matrix is identical, i.e., $J_i=\mathrm{diag}[3,2,1]\,\mathrm{kgm^2}$ for all $i\in\mathcal{N}$. Controller gains are chosen as $k_{\Omega_i}=7$, $k^\alpha_{ij}=25$, and $k^\beta_{ij}=25.1$ for any $(i,j)\in\mathcal{E}$.



Tracking errors for relative attitudes and control inputs are shown at Figure \ref{fig:err}, where the relative attitude error vectors are defined as $e_{Q_{ij}}=\frac{1}{2}((Q_{ij}^d)^\T Q_{ij}-Q_{ij}^\T Q_{ij}^d)^\vee\in\Re^3$. These illustrate good convergence rates. 

\begin{figure}
\centerline{
	\subfigure[Relative attitude error functions $\Psi_{12},\Psi_{23},\ldots,\Psi_{67}$]{\includegraphics[width=0.47\columnwidth]{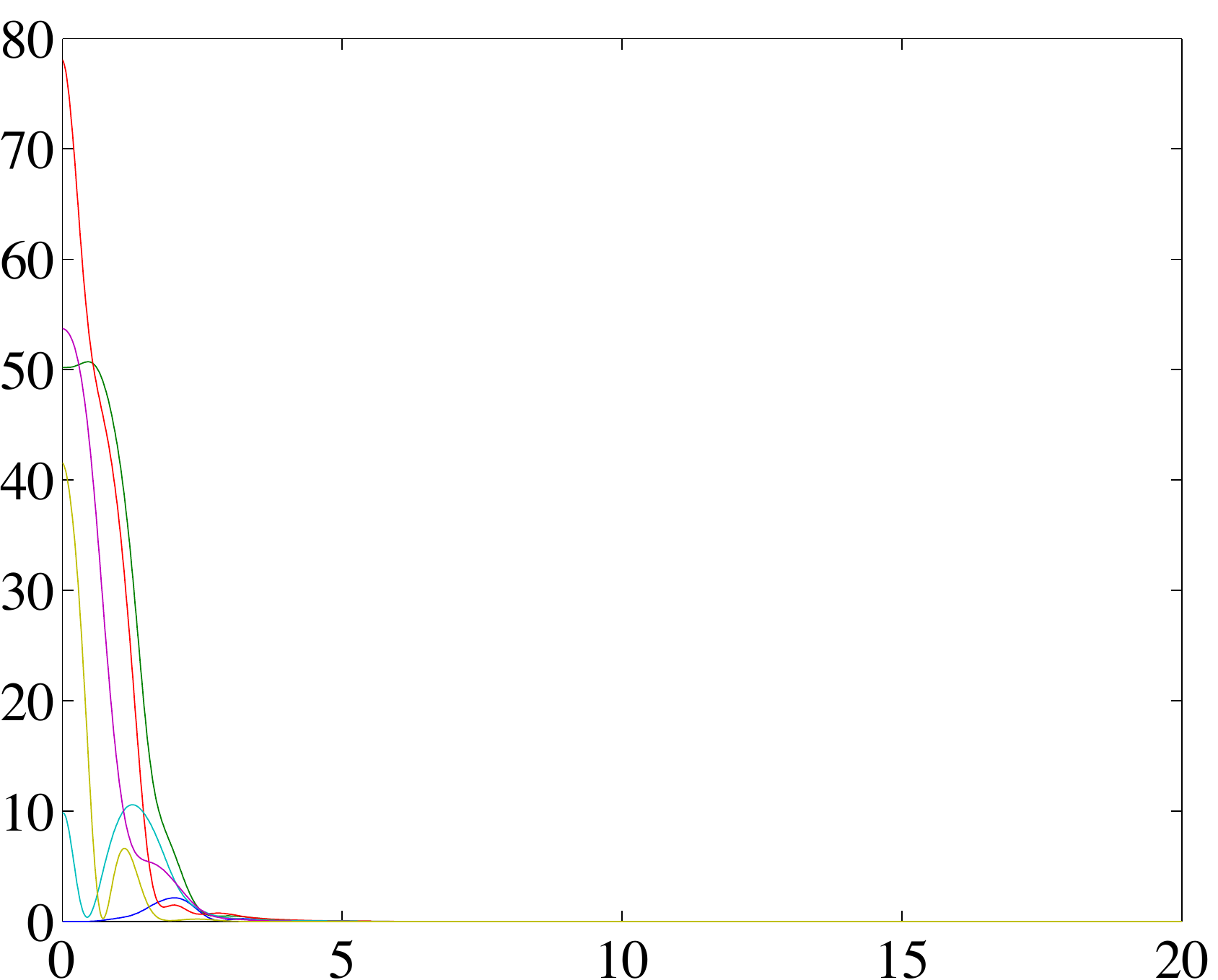}}
	\hspace*{0.03\columnwidth}
	\subfigure[Relative attitude error vectors $e_{Q_{12}},e_{Q_{23}},\ldots,e_{Q_{67}}$]{\includegraphics[width=0.47\columnwidth]{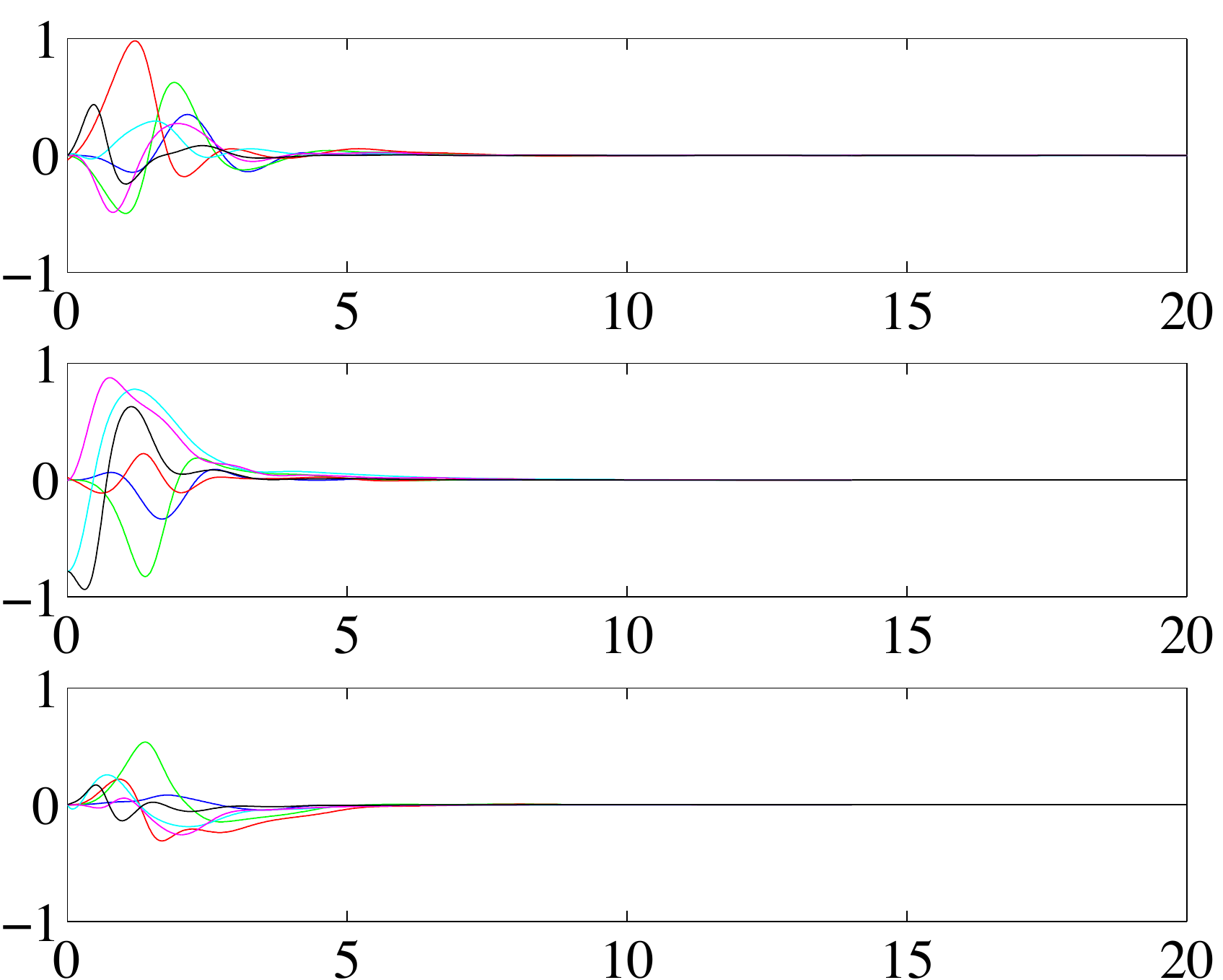}}
}
\centerline{
	\subfigure[Relative angular velocity error $e_{\Omega_1},\ldots,e_{\Omega_7}$]{\includegraphics[width=0.47\columnwidth]{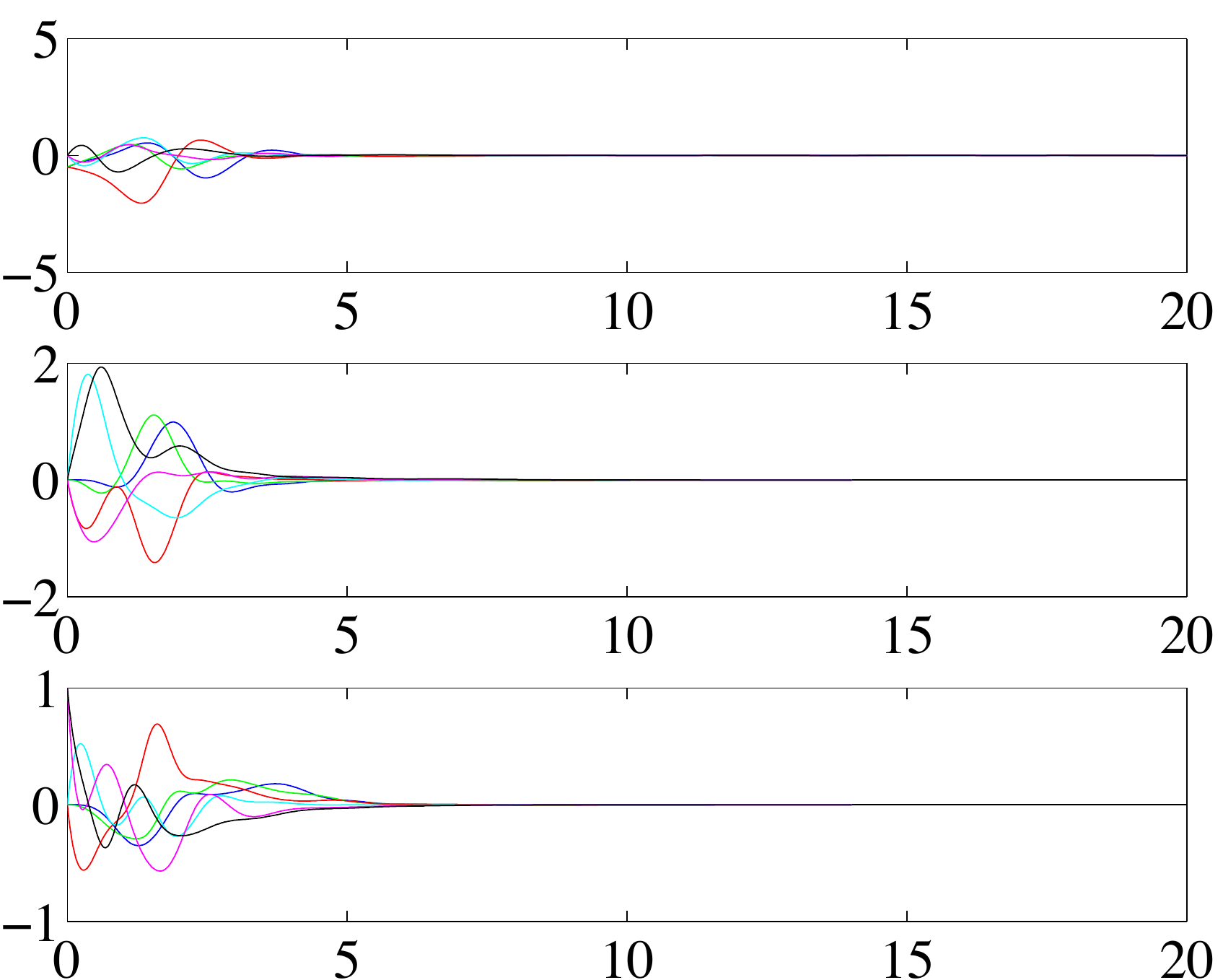}}
	\hspace*{0.03\columnwidth}
	\subfigure[Control moments $u_{1},\ldots,u_7$]{\includegraphics[width=0.47\columnwidth]{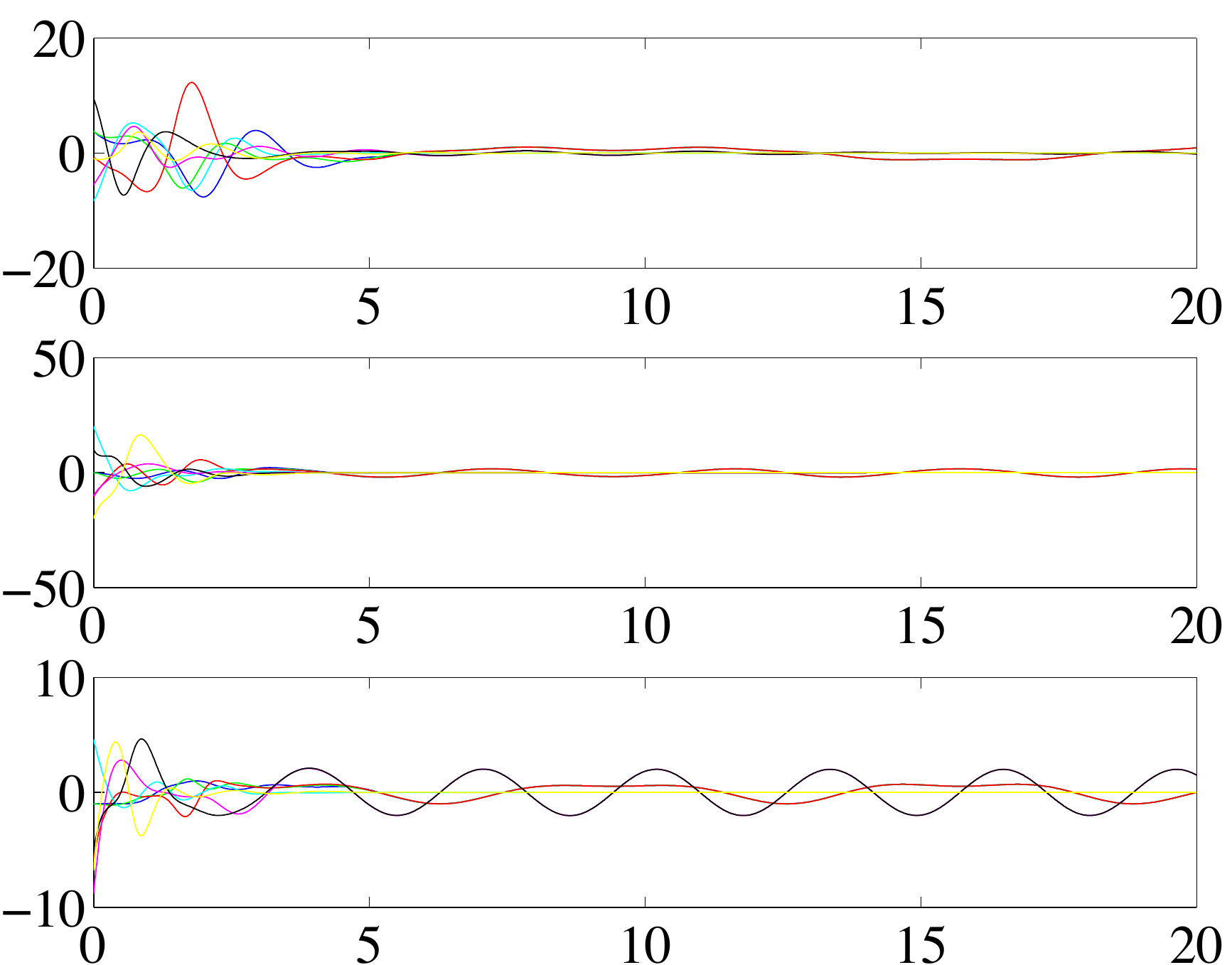}}
}
\caption{Numerical results for seven spacecraft in formation (blue, green, red, cyan, magenta, and black in ascending order)}\label{fig:err}
\end{figure}



%
%

\bibliography{ACC13}

\begin{thebibliography}{10}
\providecommand{\url}[1]{#1}
\csname url@rmstyle\endcsname
\providecommand{\newblock}{\relax}
\providecommand{\bibinfo}[2]{#2}
\providecommand\BIBentrySTDinterwordspacing{\spaceskip=0pt\relax}
\providecommand\BIBentryALTinterwordstretchfactor{4}
\providecommand\BIBentryALTinterwordspacing{\spaceskip=\fontdimen2\font plus
\BIBentryALTinterwordstretchfactor\fontdimen3\font minus
  \fontdimen4\font\relax}
\providecommand\BIBforeignlanguage[2]{{%
\expandafter\ifx\csname l@#1\endcsname\relax
\typeout{** WARNING: IEEEtran.bst: No hyphenation pattern has been}%
\typeout{** loaded for the language `#1'. Using the pattern for}%
\typeout{** the default language instead.}%
\else
\language=\csname l@#1\endcsname
\fi
#2}}

\bibitem{FaxMurITAC04}
J.~Fax and R.~Muray, ``Information flow and cooperative control of vehicle
  formations,'' \emph{IEEE Transactions on Automatic Control}, vol.~49, no.~9,
  pp. 1465--1476, 2004.

\bibitem{JadLinITAC03}
A.~Jadbabaie, J.~Lin, and A.~Morse, ``Coordination of groups of mobile
  autonomous agents using nearest neighbor rules,'' \emph{IEEE Transactions on
  Automatic Control}, vol.~48, no.~6, pp. 988--1001, 2003.

\bibitem{Mit04}
M.~Mitchell, ``{CDGPS}-based relative navigation for multiple spacecraft,''
  Ph.D. dissertation, Massachusetts Institute of Technology, 2004.

\bibitem{GarChaPIAC05}
J.~Garnham, F.~Chavez, T.~Lovell, and L.~Black, ``4-dimensional metrology
  architecture for satellite clusters using crosslinks,'' in \emph{Proceedings
  of the IEEE Aerospace Conference}, 2005, pp. 575--582.

\bibitem{KanYehIJRNC02}
W.~Kang and H.~Yeh, ``Coordinated attitutude control of multi-satellite
  systems,'' \emph{International Journal of Robust and Nonlinear Control}, vol.
  112, pp. 185--205, 2002.

\bibitem{NijRod03}
H.~Nijmeijer and A.~Rodriguez-Angeles, \emph{Synchronization of Mechanical
  Systems}.\hskip 1em plus 0.5em minus 0.4em\relax World Scientific Pub, 2003.

\bibitem{BalArkITRA98}
T.~Balch and R.~Arkin, ``Behavior-based formation control for multirobot
  teams,'' \emph{IEEE Transactions on Robotics and Automation}, vol.~14, no.~6,
  pp. 926--939, 1998.

\bibitem{BeaLawITCST01}
R.~Beard, J.~Lawton, and F.~Hadaegh, ``A coordination architecture for
  spacecraft formation control,'' \emph{IEEE Transactions on Control Systems
  Technology}, vol.~9, no.~6, pp. 777--790, 2001.

\bibitem{RenBeaPICCTA04}
W.~Ren and R.~Beard, ``Formation feedback control for multiple spacecraft via
  virtual structures,'' in \emph{Proceedings of the IEEE Conference on Control
  Theory Application}, 2004.

\bibitem{RenBeaPAGNCC02}
------, ``Virtual structure based spacecraft formation control with formation
  feedback,'' in \emph{Proceedings of the AIAA Guidance, Navigation, and
  Control Conference}, 2002, {AIAA} 2002-4963.

\bibitem{DesKakITPAMI02}
G.~Desouza and A.~Kak, ``Vision for mobile robot navigation: a survey,''
  \emph{IEEE Transactions on Pattern Analysis and Machine Intelligence},
  vol.~24, no.~2, pp. 237--267, 2002.

\bibitem{KimCraJGCD07}
S.~Kim, J.~Crassidis, Y.~Cheng, and A.~Fosbury, ``Kalman filtering for relative
  spacecraft attitude and position estimation,'' \emph{Journal of Guidance,
  Control, and Dynamics}, vol.~30, no.~1, pp. 133--143, 2007.

\bibitem{AndCraJGCD09}
M.~Andrle, J.~Crassidis, R.~Linares, Y.~Cheng, and B.~Hyun, ``Deterministic
  relative attitude determination of three-vehicle formations,'' \emph{Journal
  of Guidance, Control, and Dynamics}, vol.~43, no.~4, pp. 1077--1088, 2009.

\bibitem{LinCraJGCD11}
R.~Linares, J.~Crassidis, and Y.~Cheng, ``Constrained relative attitude
  determination for two-vehicle formations,'' \emph{Journal of Guidance,
  Control, and Dynamics}, vol.~34, no.~2, pp. 543--553, 2011.

\bibitem{LeePACC12}
T.~Lee, ``Relative attitude control of two spacecraft on {SO(3)} using
  line-of-sight observations,'' in \emph{Proceeding of the American Control
  Conference}, 2012, pp. 167--172.

\bibitem{ChaSanICSM11}
N.~Chaturvedi, A.~Sanyal, and N.~McClamroch, ``Rigid-body attitude control,''
  \emph{IEEE Control Systems Magazine}, vol.~31, no.~3, pp. 30--51, 2011.

\bibitem{BulLew05}
F.~Bullo and A.~Lewis, \emph{Geometric control of mechanical systems}, ser.
  Texts in Applied Mathematics.\hskip 1em plus 0.5em minus 0.4em\relax New
  York: Springer-Verlag, 2005, vol.~49, modeling, analysis, and design for
  simple mechanical control systems.

\bibitem{Kha96}
H.~Khalil, \emph{Nonlinear Systems}, 2nd Edition, Ed.\hskip 1em plus 0.5em
  minus 0.4em\relax Prentice Hall, 1996.

\end{thebibliography}
\bibliographystyle{IEEEtran}

\end{document}